\documentclass[11pt,reqno,a4paper]{amsart}
\usepackage[british]{babel}

\usepackage{graphicx}
\usepackage{amscd}
\usepackage{amsmath}
\usepackage{amsfonts}
\usepackage{amssymb}
\usepackage{color}
\usepackage{bbm}

\definecolor{trp}{rgb}{1,1,1}

\definecolor{red}{rgb}{1,0,.2}

\newtheorem{theorem}{Theorem}[section]
\theoremstyle{plain}

\newtheorem{cor}[theorem]{Corollary}

\newtheorem{lemma}[theorem]{Lemma}

\newtheorem{prop}[theorem]{Proposition}

\numberwithin{equation}{section}

\newcommand{\R}{\mathbb{R}}
\newcommand{\N}{\mathbb{N}}

\newcommand{\proj}{\mathrm{proj}}

\newcommand{\ii}{\mathbf{i}}

\newcommand{\mx}[1]{\underline{\underline{#1}}}

\definecolor{green}{rgb}{0.0,0.4,0.2}

\begin{document}
\title[Slicing the Sierpi\'nski gasket] {Slicing the Sierpi\'nski gasket}

\author{Bal\'azs B\'ar\'any}
\address{Bal\'azs B\'ar\'any, Department of Stochastics, Institute of Mathematics, Technical University of Budapest, 1521
Budapest, P.O.Box 91, Hungary} \email{balubsheep@gmail.com}

\author{Andrew Ferguson}
\address{Andrew Ferguson\\Department of Mathematics\\University of Bristol\\ University Walk\\Bristol\\BS8 1TW\\UK.} \email{andrew.ferguson@bris.ac.uk}

\author{K\'aroly Simon}
\address{K\'aroly Simon, Department of Stochastics, Institute of Mathematics, Technical University of Budapest, 1521
Budapest, P.O.Box 91, Hungary} \email{simonk@math.bme.hu}

 \thanks{ \indent
{\em Key words and phrases.} Hausdorff dimension, multifractal analysis, Sierpi\'nski gasket}

\begin{abstract}We investigate the dimension of intersections of the Sierpi\'nski gasket with lines.  Our first main result describes  a countable, dense set of angles that are exceptional for Marstrand's theorem.  We then provide a multifractal analysis for the set of points in the projection for which the associated slice has a prescribed dimension.
\end{abstract}
\date{\today}

\maketitle

\thispagestyle{empty}

\vspace{-0.7cm}

\section{Introduction and Statements}\label{sintro}

Let $\Delta\subset\R^2$ denote the Sierpi\'nski gasket,  i.e. the unique non-empty compact set satisfying
\[
\Delta=S_0(\Delta)\cup S_1(\Delta)\cup S_2(\Delta),
\]
where
\begin{equation}\label{esiernor}
S_0(x,y)=\left(\frac{1}{2}x,\frac{1}{2}y\right),\ S_1(x,y)=\left(\frac{1}{2}x+\frac{1}{2},\frac{1}{2}y\right),\ S_2(x,y)=\left(\frac{1}{2}x+\frac{1}{4},\frac{1}{2}y+\frac{\sqrt{3}}{4}\right).
\end{equation}

It is well known that $\dim_H\Delta=\dim_B\Delta=\frac{\log3}{\log2}=s$, where $\dim_H$ denotes the Hausdorff and $\dim_B$ denotes the box (or Minkowski) dimension. For the definition and basic properties of the box  and Hausdorff dimensions we refer the reader
to \cite{Fal2}.

We denote by $\proj_{\theta}$ the projection onto the line through the origin making angle $\theta$ with the $x$-axis.   For $a\in \proj_\theta(\Delta)$ we let $L_{\theta,a}=\{(x,y)\,:\,\proj_\theta(x,y)=a\}=\{(x,a+x\tan\theta)\,:\,x\in\mathbb{R}\}$.  The main purpose of this paper is to investigate the dimension theory of the slices $E_{\theta,a}=L_{\theta,a}\cap\Delta$. Since $\Delta$ is rotation and reflection invariant, we may assume without loss of generality that $\theta\in[0,\frac{\pi}{3})$.  In Proposition \ref{pdc1} we show that a dimension conservation principle holds: if $\nu_\theta$ denotes the projection of the normalised $\log(3)/\log(2)$-dimensional Hausdorff measure then for all $\theta\in [0,\pi)$ and $a\in{\rm proj}_\theta(\Delta)$ we have $\underline{d}_{\nu_{\theta}}(a)+\overline{\dim}_B E_{\theta,a}=s$, where $\underline{d}_{\nu_\theta}(a)$ denotes the lower local dimension of $\nu_\theta$  at $a$.  The analogous relationship between upper local dimension and lower box dimension is also proved.
 Furthermore, in Theorem \ref{ttyp} we prove that whenever $\tan\theta=\frac{\sqrt{3}p}{2q+p}$ for positive integers $p,q$, the direction $\theta $
 is exceptional in Marstrand's Theorem. More precisely,  the dimension of Lebesgue almost all slices is a constant strictly smaller than $s-1$ but the dimension for almost all slices with respect to the projected measure is another constant strictly greater than $s-1$.

Finally, we provide a multifractal analysis of the Hausdorff dimension of the slices $E_{\theta,a}$ for $\tan\theta=\frac{\sqrt{3}p}{2q+p}$ for positive integers $p,q$.  Furstenberg \cite{Fur} proved a dimension conservation principle for homogeneous sets, which in our setting corresponds to showing that
\begin{equation*}{\rm dim}_H(\Delta)=\sup\left\{\delta+{\rm dim}_H\{a\in\proj_\theta(\Delta)\,:\,{\rm dim}_H(E_{\theta,a})\geq \delta\}\right\}\end{equation*} i.e. any loss of dimension in the projection may be accounted for in the fibres $\{E_{\theta,a}\}_{a\in\proj_\theta(\Delta)}$. We remark that the results found in \cite{Fur} apply to a quite wide class of compact sets  $E\subset\mathbb{R}^n$ and for all linear maps $P:E\to\mathbb{R}^n$.
	In Theorem \ref{tspectra} we investigate the function \begin{equation*}\Gamma:\delta\mapsto {\rm dim}_H\{a\in\proj_\theta(\Delta)\,:\,{\rm dim}_H(E_{\theta,a})\geq \delta\}.\end{equation*} We prove that $\Gamma$ admits a multifractal description, in particular it is continuous, concave and may be represented as the Legendre transform of a pressure function.

For technical reasons we elect to prove our statements for the so-called right-angle Sierpi\'nski gasket $\Lambda$ which is the attractor of iterated function system
\begin{equation}\label{erightangle}
\Phi=\left\{F_0(x,y)=\left(\frac{x}{2},\frac{y}{2}\right),\ F_1(x,y)=\left(\frac{x}{2}+\frac{1}{2},\frac{y}{2}\right),\ F_2(x,y)=\left(\frac{x}{2},\frac{y}{2}+\frac{1}{2}\right)\right\},
\end{equation}
and intersections with lines of rational slope. There is a linear transformation $T$
\begin{equation}\label{etrans}
T=\left(
    \begin{array}{cc}
      1 & -\frac{\sqrt{3}}{3} \\
      0 & \frac{2\sqrt{3}}{3} \\
    \end{array}
  \right)
\end{equation}
which maps the Sierpi\'nski gasket into the right-angle Sierpi\'nsi gasket. Since an invertible linear transformation does not change the dimension of a set we state our results for the usual Sierpi\'nski gasket and for appropriate slopes. For the transformation see Figure \ref{fusutoright}.

\begin{figure}
  \includegraphics[width=150mm]{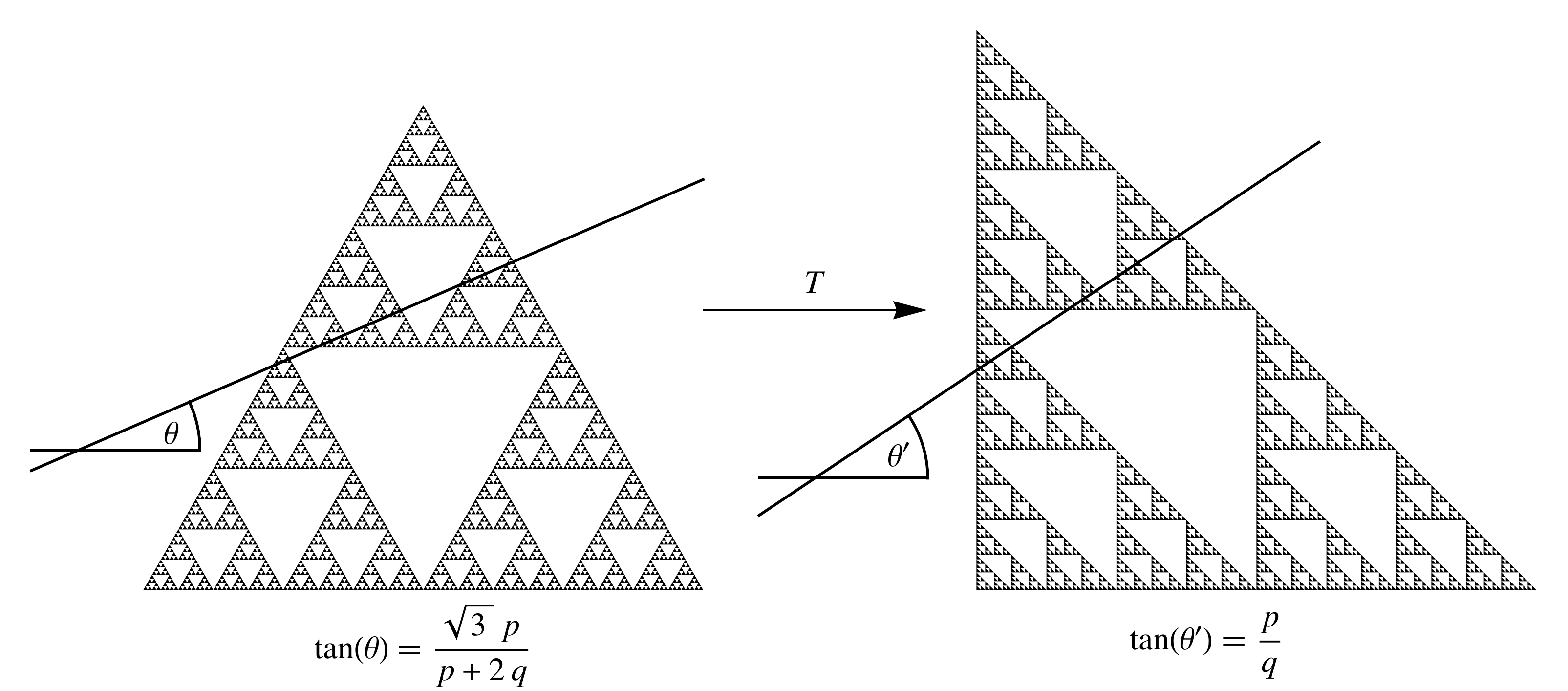}\\
  \caption{The transformation between the usual and right-angle Sierpi\'nski gasket.}\label{fusutoright}
\end{figure}

Denote by $\nu$ the unique self-similar measure satisfying
\[
\nu=\sum_{i=0}^2\frac{1}{3}\nu\circ S_i^{-1}.
\]

One may show that this measure is nothing more than the normalised $s$-dimensional Hausdorff measure restricted to $\Delta$.  We denote by $\nu_{\theta}$ the push-forward of $\nu$ under the projection $\proj_\theta$, i.e. $\nu_{\theta}=\nu\circ\proj_{\theta}^{-1}$.  Let $\Delta_{\theta}$ denote the projection of $\Delta$.

The description of typical slices is given by the following result of Marstrand (see \cite{Mar} or \cite[Theorem 10.11]{M}).
\begin{prop}[Marstrand]\label{pmar}
For Lebesgue almost every $\theta\in[0,\frac{\pi}{3})$ and $\nu_{\theta}$-almost all $a\in\Delta_{\theta}$
\[
\dim_B E_{\theta,a}=\dim_H E_{\theta,a}=s-1.
\]
\end{prop}

We define the (upper and lower) local dimension of a Borel measure $\eta$ at the point $x$ by
\[
\underline{d}_{\eta}(x)=\liminf_{r\rightarrow0}\frac{\log\eta(B_r(x))}{\log r},\ \overline{d}_{\eta}(x)=\limsup_{r\rightarrow0}\frac{\log\eta(B_r(x))}{\log r}.
\]

Manning and Simon proved a dimension conservation result for the Sierpi\'nski carpet, connecting the local dimension of the projected natural measure and the box dimension of the slices (see \cite[Proposition 4]{MS}). For the Sierpi\'nski gasket we state an analogous result.

\begin{prop}\label{pdc1}
For every $\theta\in(0,\frac{\pi}{3})$ and $a\in\Delta_{\theta}$
\begin{eqnarray}
\underline{d}_{\nu_{\theta}}(a)+\overline{\dim}_B E_{\theta,a}=s,\\
\overline{d}_{\nu_{\theta}}(a)+\underline{\dim}_B E_{\theta,a}=s.
\end{eqnarray}
\end{prop}

Feng and Hu proved in \cite[Theorem 2.12]{FH} that every self-similar measure is exact dimensional. That is, the lower and upper local-dimension coincide and this common value is almost everywhere constant.  Moreover, Young proved in \cite{Y} that this constant is the Hausdorff dimension of the measure. In other words, if $\eta$ is self-similar then
\[
\text{for $\eta$-almost all } x,\ \underline{d}_{\eta}(x)=\overline{d}_{\eta}(x)=d_{\eta}(x)=\dim_H\eta=\inf\left\{\dim_HA:\eta(A)=1\right\}.
\]
Using the above results we deduce.

\begin{cor}\label{cbox}
For every $\theta\in(0,\frac{\pi}{3})$ and $\nu_{\theta}$-almost every $a\in\Delta_{\theta}$ we have
\[
\dim_BE_{\theta,a}=s-\dim_H\nu_{\theta}\geq s-1.
\]
\end{cor}

Liu, Xi and Zhao \cite{LXZ} encoded the Box dimensions of a slice
through the Sierpi\'nski carpet for lines of rational slope in terms of the Lyapunov exponent of a random matrix product.  They then used this coding to show that for a fixed rational direction $\theta$  the Box and Hausdorff dimensions of a slice coincide and are constant for Lebesgue-almost all translations.  Moreover, this constant $\alpha(\theta )$ was shown to satisfy $\alpha(\theta )\leq s-1$, with this inequality being conjectured to be strict. This conjecture was proved by Manning and Simon
 \cite[Theorem 9]{MS}.

 We prove an analogous result for the Sierpi\'nski gasket.  In addition, we will show
 that the $\nu_\theta$-typical dimension of a slice is strictly bigger than $s-1$.
\begin{theorem}\label{ttyp}
Let $p,q\in\N$ and let us suppose that $\tan\theta=\frac{\sqrt{3}p}{2q+p}$ and $\theta\in(0,\frac{\pi}{3})$. Then there exist constants $\alpha(\theta), \beta(\theta)$ depending only on $\theta$ such that
\begin{enumerate}
  \item for Lebesgue almost all $a\in\Delta_{\theta}$\label{ttyp1}
\[
\alpha(\theta):=\dim_B E_{\theta,a}=\dim_H E_{\theta,a}<s-1,
\]
  \item for $\nu_{\theta}$-almost all $a\in\Delta_{\theta}$\label{ttyp2}
\[
\beta(\theta):=\dim_B E_{\theta,a}=\dim_H E_{\theta,a}>s-1.
\]
\end{enumerate}
\end{theorem}
A simple calculation reveals that the tangent of the set of angles in this theorem is equal
$\mathbb{Q}'=\left\{0<\sqrt[]{3}\frac{m}{n}<\sqrt[]{3}:
\mbox{ if } m \mbox{ is odd then }
n\mbox{ is odd }
\right\}$.

We remark that Theorem \ref{ttyp} shares a similarity with a result of Feng and Sidorov \cite{FS}[Theorem 3 and Proposition 4] where the  Lebesgue typical local dimension is computed for a class of self-similar measures.  The authors show that if an algebraic condition holds then the Lebesgue typical local dimension is strictly greater than one.  By Proposition \ref{pdc1} the Theorem above may be rephrased in terms of local dimensions being either strictly greater or less than one.

In \cite{Fur}, Furstenberg introduced and proved a dimension conservation formula \cite[Definition 1.1]{Fur} for homogeneous fractals (for example homotheticly self-similar sets). As a consequence of Theorem \ref{ttyp}(\ref{ttyp2}) and Corollary \ref{cbox} we state the special case of Furstenberg's dimension conservation formula for the Sierpi\'nski gasket and rational slopes.

Furstenberg in \cite[Theorem 6.2]{Fur} stated the result as an inequality but combining the result as stated with the  Marstrand Slicing Theorem (see \cite{Mar2} or \cite[Theorem 5.8]{Fal}) we see that

\begin{lemma}[Marstrand Slicing Theorem]\label{lslicing}
Let $F$ be any subset of $\R^2$, and let $E$ be a subset of the $y$-axis. If $\dim_H(F\cap L_{\theta,a})\geq t$ for all $a\in E$, then $\dim_HF\geq t + \dim_H E$.
\end{lemma}

\begin{cor}[Furstenberg]\label{cfurst}
Let $p,q\in\N$ be and let us suppose that $\tan\theta=\frac{\sqrt{3}p}{2q+p}$ and $\theta\in(0,\frac{\pi}{3})$. Then the map $\proj_{\theta}$ satisfies the dimension conservation formula \cite[Definition 1.1]{Fur} at the value $\beta(\theta)$, i.e.
\begin{equation}\label{efurst}
\beta(\theta)+\dim_H\left\{a\in\Delta_{\theta}:\dim_H E_{\theta,a}\geq\beta(\theta)\right\}=s.
\end{equation}
\end{cor}

\begin{proof}
\begin{equation*}
\begin{split}\dim_H\left\{a\in\Delta_{\theta}:\dim_H E_{\theta,a}\geq\beta(\theta)\right\}& \geq \dim_H\left\{a\in\Delta_{\theta}:\dim_B E_{\theta,a} =\dim_H E_{\theta,a}=\beta(\theta)\right\}  \\
& \geq \dim_H \nu_{\theta}=s-\beta(\theta).\end{split}
\end{equation*}
The other direction follows from Lemma \ref{lslicing}.
\end{proof}

We remark that the above argument also shows that
\[
\beta(\theta)+\dim_H\left\{a\in\Delta_{\theta}:\dim_H E_{\theta,a}=\beta(\theta)\right\}=s.
\]

The other main goal of this paper is to analyse the behaviour of the function $\Gamma:\delta\mapsto\dim_H\left\{a\in\Delta_{\theta}:\dim_H E_{\theta,a}\geq\delta\right\}$ under the assumption that $\tan\theta=\frac{\sqrt{3}p}{2q+p}$, where $p,q\in\N$ and $(p,q)=1$. For the analysis we use two matrices generated naturally by the projection and the IFS $\left\{S_0, S_1, S_2\right\}$. For simplicity, we illustrate these matrices for the right-angle gasket.

Denote the angle $\theta$ projection of $\Lambda$ to the $y$-axis by $\Lambda_{\theta}$. Then $\Lambda_{\theta}=[-\tan\theta,1]$. Consider the projected IFS of $\Phi$, i.e. \[
\phi=\left\{f_0(t)=\frac{t}{2},f_1(t)=\frac{t}{2}+\frac{1}{2},f_2(t)=\frac{t}{2}-\frac{p}{2q}\right\}.
\]
By straightforward calculations and \cite[Theorem 2.7.]{NW} we see that $\phi$ satisfies the finite type condition and therefore, the weak separation property.

Let us divide $\Lambda_{\theta}$ into $p+q$ equal intervals such that $I_k=\left[1-\frac{k}{q},1-\frac{k-1}{q}\right]$ for $k=1,\dots,p+q$. Moreover, let us divide $I_k$ for every $k$ into two equal parts. Namely, let $I_k^0=\left[1-\frac{k}{q},1-\frac{2k-1}{2q}\right]$ and $I_k^1=\left[1-\frac{2k-1}{2q},1-\frac{k-1}{q}\right]$. Let us define the $(p+q)\times(p+q)$ matrices $\mx{A}_0, \mx{A}_1$ in the following way:
\begin{equation}\label{eprojmatr}
(\mx{A}_n)_{i,j}=\sharp\left\{k\in\left\{0,1,2\right\}:f_k(I_j)=I^n_i\right\}.
\end{equation}
For example, see the case $\frac{p}{q}=\frac{2}{3}$ of the construction in Figure \ref{fex} and the matrices are
\[
\mx{A}_0=
\left(\begin{array}{ccccc}
    1 & 0 & 0 & 0 & 0 \\
    0 & 0 & 1 & 0 & 0 \\
    0 & 1 & 0 & 0 & 1 \\
    0 & 1 & 0 & 1 & 0 \\
    0 & 0 & 0 & 1 & 0 \\
\end{array}\right)\text{ and }
\mx{A}_1=\left(
           \begin{array}{ccccc}
             0 & 1 & 0 & 0 & 0 \\
             1 & 0 & 0 & 1 & 0 \\
             1 & 0 & 1 & 0 & 0 \\
             0 & 0 & 1 & 0 & 1 \\
             0 & 0 & 0 & 0 & 1 \\
           \end{array}
         \right).
\]

\begin{figure}
  \includegraphics[height=71mm]{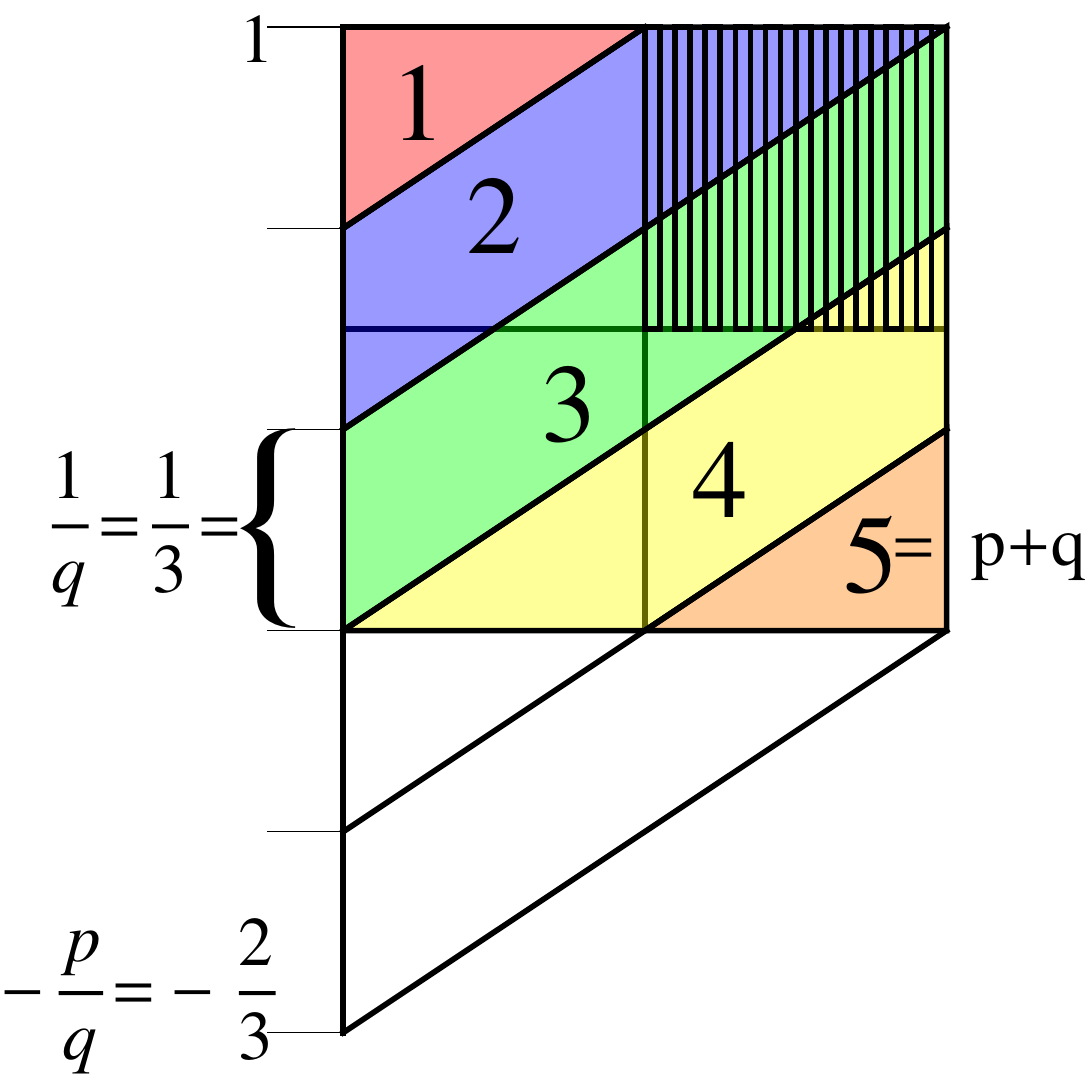}\ \ \ \ \ \ \ \ \ \ \includegraphics[height=70mm]{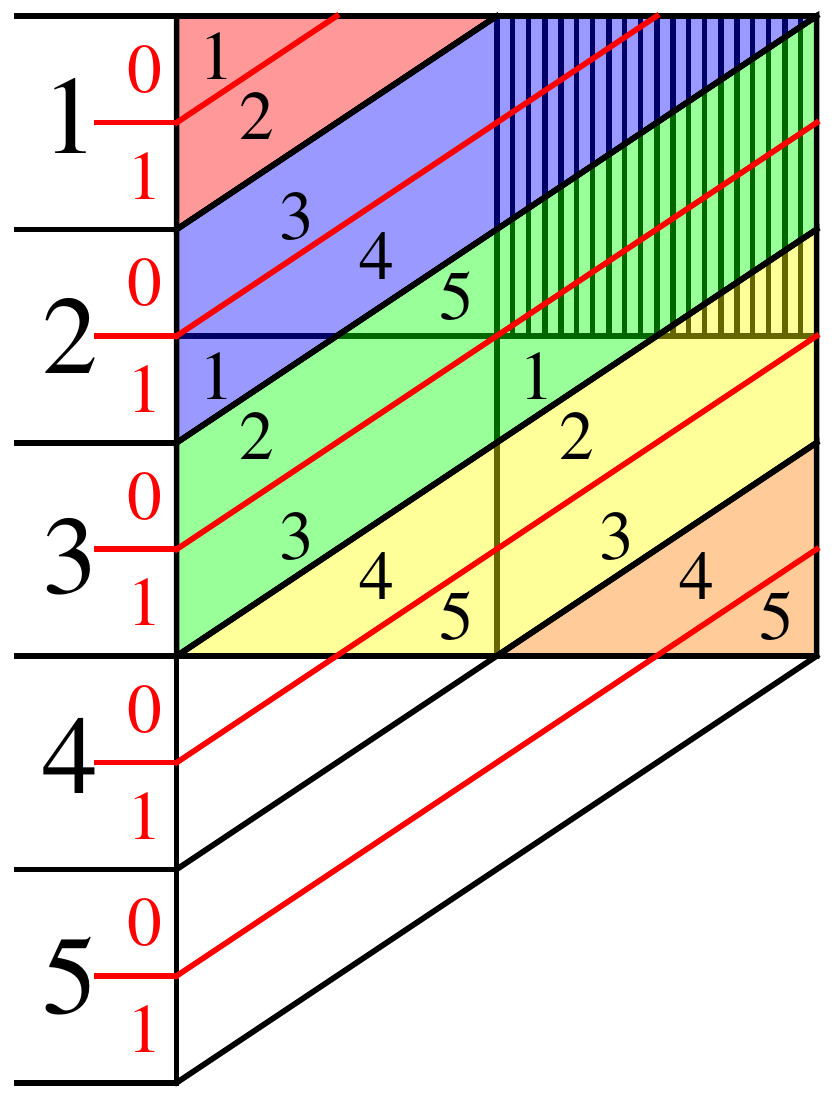}\\
  \caption{Graph of the projection and construction of matrices $\mx{A}_0, \mx{A}_1$ in the case $\frac{p}{q}=\frac{2}{3}.$}\label{fex}
\end{figure}

We note that by some simple calculations the matrices $\mx{A}_0, \mx{A}_1$ can be written in the form
\begin{multline}\label{eprojmatr2}
(\mx{A}_n)_{i,j}=1\text{ if and only if }2i+1-n\equiv j \mod p+q\text{ or }\\2q+p\geq2i+n-1\geq q+1\text{ and }2i+1-n-q\equiv j\mod p+q
\end{multline}
for $n=0,1$ and $1\leq i,j\leq p+q$. Using these matrices we are able to  explicitly express the quantities $\alpha(\theta),\beta(\theta)$.

\begin{prop}\label{pcalc}
Let $p,q\in\N$ be and let us suppose that $\tan\theta=\frac{\sqrt{3}p}{2q+p}$ and $\theta\in(0,\frac{\pi}{3})$. Moreover, let $\alpha(\theta)$ and $\beta(\theta)$ be as in Theorem \ref{ttyp}. Then
\begin{equation*}
\begin{split}\alpha(\theta)&=\frac{1}{\log2}\lim_{n\rightarrow\infty}\frac{1}{n}\sum_{\xi_1,\dots,\xi_n=0}^1\frac{1}{2^n}\log\underline{e}\mx{A}_{\xi_1}\cdots\mx{A}_{\xi_n}\underline{e},\\
\beta(\theta)& =\frac{1}{\log2}\lim_{n\rightarrow\infty}\frac{1}{n}\sum_{\xi_1,\dots,\xi_n=0}^1\frac{1}{3^n}\underline{e}\mx{A}_{\xi_1}\cdots\mx{A}_{\xi_n}\underline{p}\log\left(\underline{e}\mx{A}_{\xi_1}\cdots\mx{A}_{\xi_n}\underline{p}\right),
\end{split}\end{equation*}
where $\underline{e}=(1,\cdots,1)$ and $(\mx{A}_0+\mx{A}_1)\underline{p}=3\underline{p}$.
\end{prop}
The proof of Proposition \ref{pcalc} will follow from the proof of Theorem \ref{ttyp}.  In order to obtain further information on the nature of the function $\Gamma(\delta)$ we will employ the theory of multifractal analysis for products of non-negative matrices \cite{F1,F2,FL2}.  Let $P(t)$ denote the pressure function which is defined as
\begin{equation}\label{emxpressure}
P(t)=\lim_{n\rightarrow\infty}\frac{1}{n}\log\sum_{\xi_1,\dots,\xi_n=0}^1\left(\underline{e}\mx{A}_{\xi_1}\cdots\mx{A}_{\xi_n}\underline{e}\right)^t
\end{equation}
and let us define \[
b_{\min}=\lim_{t\rightarrow-\infty}\frac{P(t)}{t},\ b_{\max}=\lim_{t\rightarrow\infty}\frac{P(t)}{t}.
\]

\begin{prop}\label{cspectra}
Let $p,q\in\N$ and let us suppose that $\tan\theta=\frac{\sqrt{3}p}{2q+p}$ and $\theta\in(0,\frac{\pi}{3})$. Then
\begin{enumerate}
    \item $\dim_H\left\{a\in\Delta_{\theta}:\dim_BE_{\theta,a}=\alpha\right\}=\inf_t\left\{-\alpha t+\frac{P(t)}{\log2}\right\}$ for $b_{\min}\leq\alpha\leq b_{\max}$.\label{tspectra1}
    \item $\dim_H\left\{a\in\Delta_{\theta}:d_{\nu_{\theta}}(a)=\alpha\right\}=\inf_t\left\{-(s-\alpha) t+\frac{P(t)}{\log2}\right\}$ for $s-b_{\max}\leq\alpha\leq s-b_{\min}$.\label{tspectra4}
\end{enumerate}
Both of the functions are concave and continuous.
\end{prop}

\begin{proof}
Proposition \ref{cspectra}(\ref{tspectra4}) follows immediately from \cite[Theorem 1.1]{FL}, \cite[Theorem 1.2]{FL}. Proposition \ref{cspectra}(\ref{tspectra1}) follows from combining the dimension conservation principle Proposition \ref{pdc1} with Proposition \ref{cspectra}(\ref{tspectra4}).
\end{proof}

We note that Proposition \ref{cspectra}(\ref{tspectra1}) may derived by applying \cite{F2} to the matrices $\mx{A}_0, \mx{A}_1$.  We describe this derivation in Section \ref{stspectra}.

\begin{theorem}\label{tspectra}
Let $p,q\in\N$ and let us suppose that $\tan\theta=\frac{\sqrt{3}p}{2q+p}$ and $\theta\in(0,\frac{\pi}{3})$. Then
\begin{enumerate}
    \item$\Gamma(\delta)=\dim_H\left\{a\in\Delta_{\theta}:\dim_H E_{\theta,a}\geq\delta\right\}=\inf_{t>0}\left\{-\delta t+\frac{P(t)}{\log2}\right\}$ if $b_{\max}\geq\delta>\alpha(\theta)$ and $\Gamma(\delta)=1$ if $\delta\leq\alpha(\theta)$. The function $\Gamma$ is decreasing and continuous.\label{tspectra2}
    \item For every $b_{\max}\geq\delta\geq\alpha(\theta)$, $\chi(\delta)=\dim_H\left\{a\in\Delta_{\theta}:\dim_H E_{\theta,a}=\delta\right\}=\inf_{t>0}\left\{-\delta t+\frac{P(t)}{\log2}\right\}$. The function $\chi$ is decreasing and continuous.\label{tspectra3}
\end{enumerate}
\end{theorem}

For an example of the function $\delta\mapsto\dim_H\left\{a\in\Delta_{\theta}:\dim_H E_{\theta,a}=\delta\right\}$ with $\tan\theta=\frac{\sqrt{3}}{3}$ in the usual Sierpi\'nski gasket case, see Figure \ref{fspect}.

\begin{figure}
  \includegraphics[width=80mm]{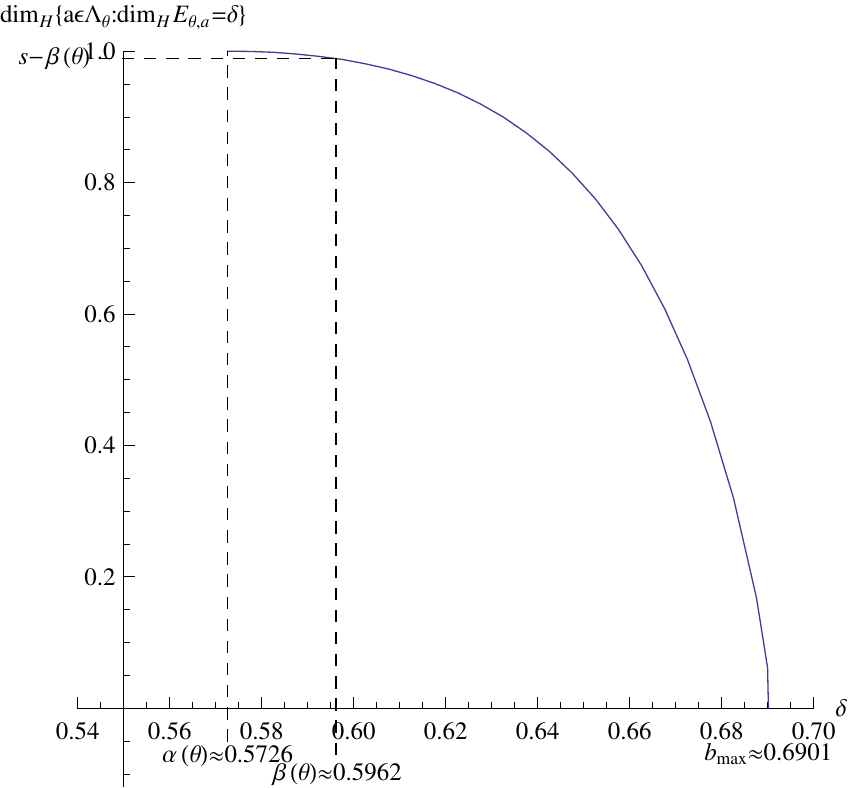}\\
  \caption{The graph of the function $\delta\mapsto\dim_H\left\{a\in\Delta_{\theta}:\dim_H E_{\theta,a}=\delta\right\}$ of the case $\frac{p}{q}=1$.}\label{fspect}
\end{figure}

The organisation of the paper is as follows: in Section \ref{sprop} we prove \ref{pdc1}.  Section \ref{sttyp} contains proof of Proposition \ref{ttyp}.  Finally, in Section \ref{stspectra} we prove Theorem \ref{tspectra}.

\section{Proof of Proposition \ref{pdc1}}\label{sprop}

In this section we prove Proposition \ref{pdc1}.  The method is adapted from \cite[Proposition 4]{MS} where an analogous result is proved for the Sierpi\'nski carpet.

We first introduce notation that will be fixed for the remainder of the paper. Let $S_0, S_1, S_2$ be as in (\ref{esiernor}), moreover let $\Sigma=\left\{0,1,2\right\}^{\N}$ and $\Sigma^*=\bigcup_{n=0}^{\infty}\left\{0,1,2\right\}^n$. Write $\sigma:\Sigma\mapsto\Sigma$ for the left shift operator. Moreover, let $\Pi:\Sigma\mapsto\Delta$ the natural projection. That is, for every $\ii=(i_1i_2\cdots)\in\Sigma$
\[
\Pi(\ii)=\lim_{n\rightarrow\infty}S_{i_1}\circ S_{i_2}\circ\cdots\circ S_{i_n}(0).
\]
Let $\mu$ be the equally distributed Bernoulli measure on $\Sigma$. That is, for every $\underline{i}\in\Sigma^*$ the measure of $[\underline{i}]=\left\{\ii:\ii=\underline{i}\omega\right\}$ is $\mu([\underline{i}])=3^{-\left|\underline{i}\right|}$, where $\left|\underline{i}\right|$ denotes the length of $\underline{i}$. Then $\nu=\Pi^*\mu=\mu\circ\Pi^{-1}$.

For simplicity we denote by $\Delta_{i_1\cdots i_n}=S_{i_1}\circ\cdots\circ S_{i_n}(\Delta)$. Let us call the $n$'th level ``good sets'' of $a\in\Delta_{\theta}$ the set of $(i_1\cdots i_n)$ such that $\Delta_{i_1\cdots i_n}$ intersects the set $E_{\theta,a}$. More precisely,

\begin{equation}\label{egoodset}
G_n(\theta,a)=\left\{(i_1\cdots i_n):\Delta_{i_1\cdots i_n}\cap E_{\theta,a}\neq\emptyset\right\}.
\end{equation}

\begin{lemma}\label{lbdgs}
For every $\theta\in[0,\frac{\pi}{3})$ and $a\in\Delta_{\theta}$
\[
\underline{\dim}_B E_{\theta,a}=\liminf_{n\rightarrow\infty}\frac{\log\sharp G_n(\theta,a)}{n\log2}\text{ and } \overline{\dim}_B E_{\theta,a}=\limsup_{n\rightarrow\infty}\frac{\log\sharp G_n(\theta,a)}{n\log2},
\]
where $\sharp G_n(\theta,a)$ denotes the cardinality of $G_n(\theta,a)$.
\end{lemma}

\begin{proof}
Let us denote  the minimal number of intervals with length $r$ covering the set $E_{\theta,a}$ by $N_r(\theta,a)$. It is easy to see that
\begin{equation}\label{ebll}
N_{2^{-n}}(\theta,a)\leq\sharp G_n(\theta,a).
\end{equation}

On the other hand, for a minimal cover of $E_{\theta,a}$ with intervals of side length $2^{-n}$ every such interval will intersect an element of $G_n(\theta,a)$.  Further, every element of $G_n(\theta,a)$ will intersect some element of this minimal cover. Finally, we observe that for every interval with side length $2^{-n}$ there are at most $\left\lceil\frac{4\sqrt{3}(2+\pi)}{3}\right\rceil$ cylinders in $G_n(\theta,a)$ which images under $\Pi$ intersect the interval. Hence,
\begin{equation}\label{ebul}
\sharp G_n(\theta,a)\leq\left\lceil\frac{4\sqrt{3}(2+\pi)}{3}\right\rceil N_{2^{-n}}(\theta,a).
\end{equation}

The equations (\ref{ebll}) and (\ref{ebul}) imply the statement of the lemma.
\end{proof}

\begin{proof}[Proof of Proposition \ref{pdc1}]
Let $\theta\in(0,\frac{\pi}{3})$ and let us take a point $a\in\Delta_{\theta}$. Take the $C(\theta)2^{-n}$ neighbourhood of $a$, where $C(\theta)=\frac{1}{2}\min\left\{\tan\theta,\cos(\theta+\frac{\pi}{6})\right\}$. Then
\[
\nu_{\theta}(B_{C(\theta)2^{-n}}(a))=\nu(B_{\cos\theta C(\theta)2^{-n}}(L_{\theta,a}))\geq\nu\left(\bigcup_{\underline{i}\in G_{n-c(\theta)}}\Delta_{\underline{i}}\right)=3^{-n+c(\theta)}\sharp G_{n-c(\theta)}(\theta,a),
\]
where $c(\theta)=\frac{\log\left(\cos\theta C(\theta)\right)}{\log2}$. Taking logarithms and dividing by $-n\log2$ yields
\[
\frac{\log(\nu_{\theta}(B_{C(\theta)2^{-n}}(a)))}{-n\log2}\leq\frac{(n-c(\theta))\log3}{n\log2}+\frac{\log\sharp G_{n-c(\theta)}(\theta,a)}{-n\log2}.
\]
Taking limit inferior and limit superior and using Lemma \ref{lbdgs} we obtain
\begin{equation}\label{edcub}
\begin{split}
&\underline{d}_{\nu_{\theta}}(a)+\overline{\dim}_B E_{\theta,a}\leq s,\\
&\overline{d}_{\nu_{\theta}}(a)+\underline{\dim}_B E_{\theta,a}\leq s.
\end{split}
\end{equation}

For the reverse inequality we have to introduce the so called ``bad'' sets which do not intersect $E_{\theta,a}$ but intersect its neighbourhood. That is,
\[
R_n(\theta,a)=\left\{(i_1\cdots i_n):\Delta_{i_1\cdots i_n}\cap E_{\theta,a}=\emptyset\text{ and }\Delta_{i_1\cdots i_n}\cap B_{\cos\theta C(\theta)2^{-n}}(L_{\theta,a})\neq\emptyset\right\}.
\]
Then
\[
\nu_{\theta}(B_{C(\theta)2^{-n}}(a))=\nu(B_{\cos\theta C(\theta)2^{-n}}(L_{\theta,a}))\leq3^{-n}\left(\sharp R_n(\theta,a)+\sharp G_n(\theta,a)\right).
\]
It is enough to prove that $\sharp R_n(\theta,a)$ is less than or equal to $\sharp G_n(\theta,a)$ up to a multiplicative constant.

\begin{figure}
\begin{center}
  \includegraphics[width=10cm]{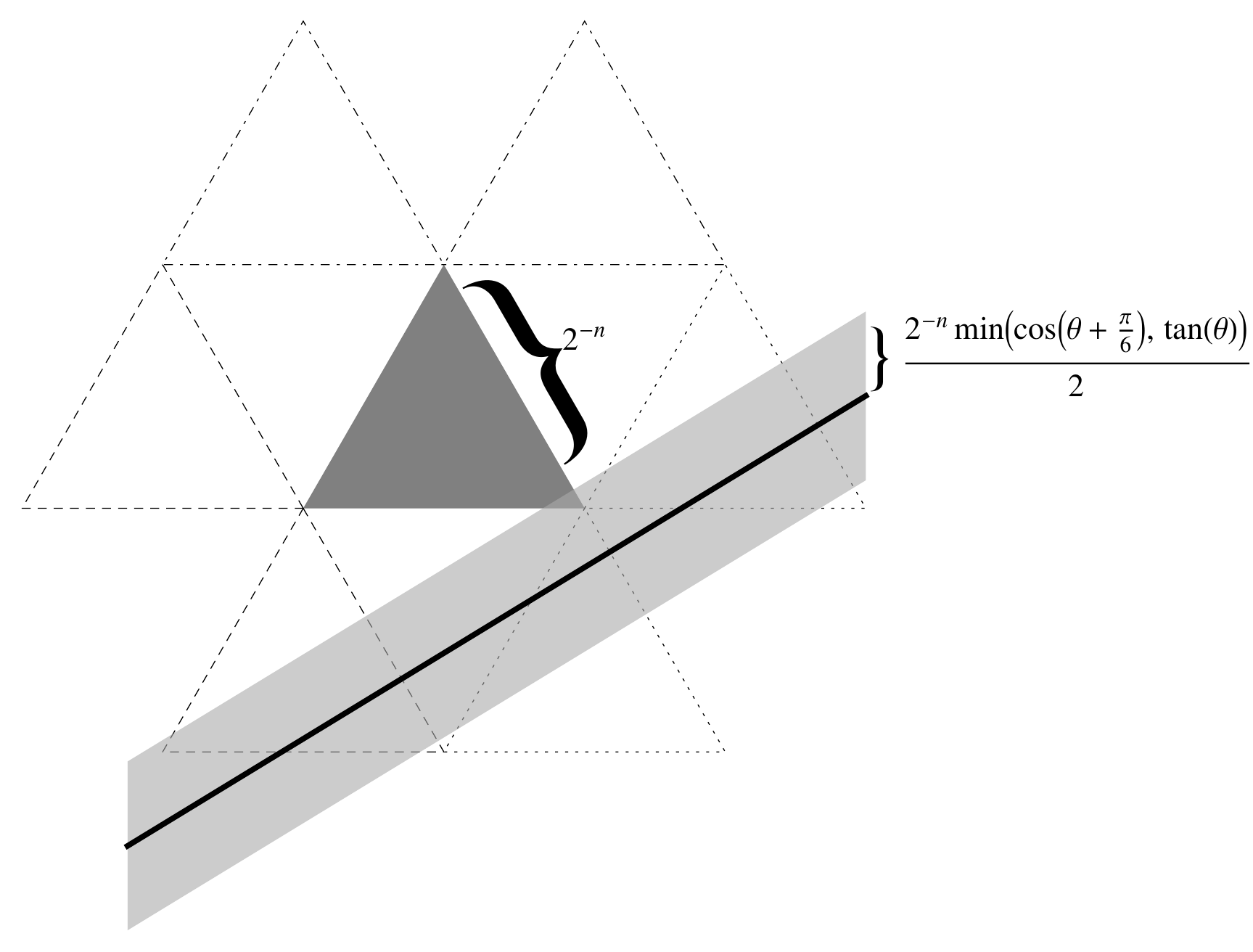}\\
  \caption{A ``bad'' set of the Sierpi\'nski gasket}\label{fsier}
\end{center}
\end{figure}

Let $\Delta_{\underline{i}}$ be an arbitrary $n$'th level cylinder set of $\Delta$. It is easy to see that if $\Delta_{\underline{i}}$ is not one of the corners of $\Delta$ then every corner of $\Delta_{\underline{i}}$ connects to another $n$'th level cylinder set, see Figure \ref{fsier}. We note that the constant $C(\theta)$ is chosen in the way that if the $\cos\theta C(\theta)2^{-n}$ neighbourhood of the line $L_{\theta,a}$ intersects a cylinder but not the line itself intersects it (that is it is a ``bad'' set) then the line intersects the closest neighbour of the cylinder. Therefore, for every $\underline{i}\in R_n(\theta,a)$ there exists at least one $\underline{j}\in G_n(\theta,a)$ such that $\Delta_{\underline{i}}$ and $\Delta_{\underline{j}}$ are connected to each other (by the choice of $C(\theta)$). Moreover, a cylinder set can be connected to at most $6$ other cylinder sets. Therefore, $R_n(\theta,a)\leq6G_n(\theta,a)$.

Applying that, we have
\[
\nu_{\theta}(B_{C(\theta)2^{-n}}(a))\leq3^{-n}7\#G_n(\theta,a).
\]
Taking logarithms, dividing by $-n\log2$ and taking limit inferior and limit superior we get by Lemma \ref{lbdgs}
\begin{equation}\label{edclb}
\begin{split}
&\underline{d}_{\nu_{\theta}}(a)+\overline{\dim}_B E_{\theta,a}\geq s,\\
&\overline{d}_{\nu_{\theta}}(a)+\underline{\dim}_B E_{\theta,a}\geq s.
\end{split}
\end{equation}

The inequalities (\ref{edcub}) and (\ref{edclb}) imply the statements.
\end{proof}

\section{Proof of Theorem \ref{ttyp}}\label{sttyp}

In this Section we prove Theorem \ref{ttyp}, that is for $\tan(\theta)\in\mathbb{Q}$ the angle $\theta$ is an exceptional direction in Marstrand's theorem.  We encode the box dimension of a slice $E_{\theta,a}$ using the matrices $\mx{A}_0,\mx{A}_1$.  This coding first appeared in \cite{LXZ}.  We then show that the Lebesgue-typical dimension of a slice is strictly less than $s-1$ by applying a result of Manning and Simon \cite[Theorem 9]{MS}.   Finally, we show that the $\nu_\theta$-typical dimension of a slice is strictly bigger than $s-1$.

For the rest of the paper we will work with the right-angle Sierpi\'nski gasket $\Lambda$ and for rational slopes. 

For the rest of the paper we assume that $\theta\in(0,\frac{\pi}{2})$ such that $\tan\theta=\frac{p}{q}$ where $p,q\in\N$ and the greatest common divisor is $1$. (This is equivalent with the choice $\theta\in(0,\frac{\pi}{3})$ for $\Delta$.)

\begin{lemma}\label{lmxbox}
Let $\theta$ and $a\in\Lambda_{\theta}$ be such that $\tan\theta=\frac{p}{q}$ and
\[
a=1-\frac{k-1}{q}-\frac{1}{q}\sum_{i=1}^{\infty}\frac{\xi_i}{2^i}
\]
then
\[
\underline{\dim}_B E_{\theta,a}=\liminf_{n\rightarrow\infty}\frac{\log\underline{e}_k\mx{A}_{\xi_1}\cdots\mx{A}_{\xi_n}\underline{e}}{n\log2}\text{ and }\overline{\dim}_B E_{\theta,a}=\limsup_{n\rightarrow\infty}\frac{\log\underline{e}_k\mx{A}_{\xi_1}\cdots\mx{A}_{\xi_n}\underline{e}}{n\log2},
\]
where $\underline{e}_k$ is the $k$'th element of the natural basis of $\R^{p+q}$ and $\underline{e}=\sum_{k=1}^{p+q}\underline{e}_k$.
\end{lemma}

\begin{proof}
By the definition of the matrices $\mx{A}_0, \mx{A}_1$ it is easy to see that for every $n\geq1$ and $\xi_1,\dots,\xi_n\in\left\{0,1\right\}$ we have
\[
\left(\mx{A}_{\xi_1}\cdots\mx{A}_{\xi_n}\right)_{i,j}=\sharp\left\{\underline{i}\in\left\{0,1\right\}^n:f_{\underline{i}}(I_j)=I_i^{\xi_1,\dots,\xi_n}\right\},
\]
where $I_i^{\xi_1,\dots,\xi_n}$ denotes the interval $[1-\frac{i-1}{q}-\frac{1}{q}\sum_{l=1}^n\frac{\xi_l}{2^l}-\frac{1}{q2^{n}},1-\frac{i-1}{q}-\frac{1}{q}\sum_{l=1}^n\frac{\xi_l}{2^l}]$. Therefore
\[
\underline{e}_k\mx{A}_{\xi_1}\cdots\mx{A}_{\xi_n}\underline{e}=\sharp\left\{\underline{i}\in\left\{0,1\right\}^n:\text{there exists a }1\leq j\leq p+q\text{ such that }f_{\underline{i}}(I_j)=I_k^{\xi_1,\dots,\xi_n}\right\}.
\]
For every $I_k^{\xi_1,\dots,\xi_n}$ and every $(i_1,\dots,i_n)$ such that there exists a $1\leq j\leq p+q$ such that $f_{i_1,\dots,i_n}(I_j)=I_k^{\xi_1,\dots,\xi_n}$, then $I_k^{\xi_1,\dots,\xi_n}\subseteq\proj_{\theta}\Lambda_{i_1,\dots,i_n}$. This implies that for every $a\in I_k^{\xi_1,\dots,\xi_n}$
\[
\underline{e}_k\mx{A}_{\xi_1}\cdots\mx{A}_{\xi_n}\underline{e}\leq\sharp G_n(\theta,a).
\]

On the other hand for every $a\in\proj_{\theta}\Lambda$ if $a\in\mathrm{int}(I_k^{\xi_1,\dots,\xi_n})$ then for every $(i_1,\dots,i_n)\in G_n(\theta,a)$ there exists a $1\leq j\leq p+q$ such that $f_{i_1,\dots,i_n}(I_j)=I_k^{\xi_1,\dots,\xi_n}$. If $a\in\partial(I_k^{\xi_1,\dots,\xi_n})$ then for every $(i_1,\dots,i_n)\in G_n(\theta,a)$ there exists a $(i_1',\dots,i_n')\in G_n(\theta,a)$ and a $1\leq j\leq p+q$ such that $f_{i_1',\dots,i_n'}(I_j)=I_k^{\xi_1,\dots,\xi_n}$ as well as $\Lambda_{i_1,\dots,i_n}$ and $\Lambda_{i_1',\dots,i_n'}$ are connected or equal.
Since for every cylinder set can be connected to at most three other cylinder sets we have for any $a\in I_k^{\xi_1,\dots,\xi_n}$ that
\[
\sharp G_n(\theta,a)\leq3\underline{e}_k\mx{A}_{\xi_1}\cdots\mx{A}_{\xi_n}\underline{e}.
\]
The proof is completed by Lemma \ref{lbdgs}.
\end{proof}

One of the main properties of the matrices $\mx{A}_0, \mx{A}_1$ is stated in the following proposition.

\begin{prop}\label{pmain}
Let $p,q$ be integers such that the greatest common divisor is $1$, and let $\mx{A}_0$ and $\mx{A}_1$ be defined as in (\ref{eprojmatr}) (or equivalently as in (\ref{eprojmatr2})). Then there exists $n_0\geq1$ and a finite sequence $(\xi_1,\dots,\xi_{n_0})\in\left\{0,1\right\}^{n_0}$ such that every element of $\mx{A}_{\xi_1}\cdots\mx{A}_{\xi_{n_0}}$ is strictly positive.

Moreover, for every $n\geq1$
\begin{multline}\label{epmain}
\sharp\left\{(\xi_1,\dots,\xi_n)\in\left\{0,1\right\}^n:\exists1\leq i,j\leq p+q\text{ such that }\left(\mx{A}_{\xi_1,\dots,\xi_n}\right)_{i,j}=0\right\}\leq\\\sum_{l=0}^{(p+q-1)(p+q)-1}\binom{n}{l}2^l.
\end{multline}
\end{prop}

We divide the proof of Proposition \ref{pmain} into the following three lemmas.

\begin{lemma}\label{lmxbasic}
Let $p,q$ be integers such that the greatest common divisor is $1$, and let $\mx{A}_0$ and $\mx{A}_1$ be defined as in (\ref{eprojmatr}). Then there are at least one and at most two $1$'s in each column and in each row of $\mx{A}_n$. Moreover, the sum of each column of $\mx{A}_0+\mx{A}_1$ is three.
\end{lemma}

The proof follows immediately from the definition.

\begin{lemma}\label{lmxatleast}
Let $p,q$ be integers such that the greatest common divisor is $1$, and let $\mx{A}_0$ and $\mx{A}_1$ be defined as in (\ref{eprojmatr}) and (\ref{eprojmatr2}). Then for every $1\leq m\leq p+q$ distinct columns $1\leq j_1,\dots,j_m\leq p+q$ and every $n=0,1$ there exist $m$ distinct rows $1\leq i_1,\dots,i_m\leq p+q$ such that $\left(\mx{A}_n\right)_{i_k,j_k}=1$ for every $k=1,\dots,m$. Note that $i_1,\dots,i_m$ may depend on $n$.

\end{lemma}

\begin{proof}

If $p+q$ is odd then for any $j_k$ there exists a unique $i_k$ such that $2i_k-1+n\equiv j_k\mod p+q$ and, by (\ref{eprojmatr2}), $\left(\mx{A}_n\right)_{i_k,j_k}=1$. Moreover, if $j_k\neq j_{k'}$ then $i_k\neq i_{k'}$. This implies the statement of the lemma.

Now, let us assume that $p+q$ is even. Further, assume that there are two non-zero elements $j_1, j_2$ in the row $i_1$. Then
\[
2i_1-1+n\equiv j_1\mod p+q\text{ and }2i_1-1+n-q\equiv j_2\mod p+q.
\]
It is easy to see that every element of the column $j_2$ is $0$ except $(i_1,j_2)$. Moreover, there exists $1\leq i_1'\leq p+q$ such that $2i_1'-1+n\equiv j_1\mod p+q$. In this case, every element of the row $i_1'$ is $0$ except $(i_1',j_1)$. Otherwise, if there would be $j_3\neq j_1$ such that $2i_1'-1+n-q\equiv j_3\mod p+q$ then $j_3\equiv j_1-q\equiv j_2\mod p+q$, but every element of the column $j_2$ is zero except $(i_1,j_2)$, which is a contradiction.
Therefore, for $\mx{A}_n$, $n=0,1$ and for every $m$ distinct columns $j_1,\dots,j_m$ there are at least $m$ distinct rows $i_1,\dots,i_m$ such that $\left(\mx{A}_n\right)_{i_k,j_k}=1$.
\end{proof}

\begin{lemma}\label{lmxgrowth}
Let $p,q$ be integers such that the greatest common divisor is $1$, and let $\mx{A}_0$ and $\mx{A}_1$ be defined as in (\ref{eprojmatr}) and in (\ref{eprojmatr2}). Then for every $1\leq m<p+q$ distinct columns $1\leq j_1,\dots,j_m\leq p+q$ there exists an $n\in\left\{0,1\right\}$ and at least $m+1$ distinct rows $1\leq i_1,\dots,i_{m+1}\leq p+q$ such that $\left(\mx{A}_n\right)_{i_k,j_k}=1$ for $k=1,\dots,m$ and there exists a $j\in\left\{j_1,\dots,j_m\right\}$ such that $\left(\mx{A}_n\right)_{i_{m+1},j}=1$.
\end{lemma}

\begin{proof}
We argue by contradiction. Let us fix the $m$ distinct columns $1\leq j_1,\dots,j_m\leq p+q$. By Lemma \ref{lmxbasic} in every column there are at least one and at most two ``$1$'' elements and by Lemma \ref{lmxatleast} there are at least $m$ different rows $1\leq i_1,\dots, i_m\leq p+q$ in $\mx{A}_0$ and at least $m$ different rows $1\leq s_1,\dots, s_m\leq p+q$ in $\mx{A}_1$ such that $\left(\mx{A}_0\right)_{i_k,j_k}=1$ and $\left(\mx{A}_1\right)_{s_k,j_k}=1$. To get a contradiction we assume that
\begin{equation}\label{a1}
 \tag{\textbf{A1}}
\forall i\not\in \left\{i_1,\dots ,i_m\right\},
\forall s\not\in \left\{s_1,\dots ,s_m\right\},
\forall k:\quad
\left(\mx{A}_0\right)_{i,j_k}=0,\
\left(\mx{A}_1\right)_{s,j_k}=0.
\end{equation}
By Lemma \ref{lmxbasic} the matrix $\mx{A}_0+\mx{A}_1$ has in each column exactly $3$ non-zero elements. Therefore we can assume without loss of generality that there is an $0\leq l\leq m$ such that in $\mx{A}_0$ the columns $j_1,\dots,j_l$ and in $\mx{A}_1$ the columns $j_{l+1},\dots,j_m$ contain two non-zero elements. Namely, there are $l$ distinct rows $1\leq i_1',\dots,i_l'\leq$ and $m-l$ distinct rows $1\leq s_{l+1}',\dots,s_m'\leq p+q$ such that $\left(\mx{A}_0\right)_{i_k',j_k}=1$ for $k=1,\dots,l$ and $\left(\mx{A}_1\right)_{s_k',j_k}=1$ for $k=l+1,\dots,m$. Moreover, by our assumption (\textbf{A1}) and Lemma \ref{lmxatleast}, for every $i_k'$ there exists a $i_{t_k}$ such that $l+1\leq t_k\leq m$ and $i_k'=i_{t_k}$.  Similarly, for every $s_k'$ there exists a $s_{t_k}$ such that $1\leq t_k\leq l$ and $s_k'=s_{t_k}$.

Let us define now a directed graph $G(V,E)$ such that the vertices are $V=\left\{j_1,\dots,j_m\right\}$ and there is an edge $j_k\rightarrow j_n$ if and only if $s_k'=s_n$ or $i_k'=i_n$. It is easy to see that
\begin{equation}\label{eedge}
j_k\rightarrow j_n\Longleftrightarrow\left\{\begin{array}{crcr} j_n-q\equiv j_k\mod p+q & \text{if }p+q\text{ is odd}\\
                                                                j_k-q\equiv j_n\mod p+q & \text{if }p+q\text{ is even.}\end{array}\right.
\end{equation}

Since from every vertex of $G$ there is an edge pointing out, there is a cycle $j_{n_1}\rightarrow j_{n_2}\rightarrow\cdots\rightarrow j_{n_t}\rightarrow j_{n_1}$, where $1\leq t\leq m$. By (\ref{eedge}) we have
\[
\begin{split}
&j_{n_1}\equiv j_{n_2}-q\equiv\cdots\equiv j_{n_t}-(t-1)q\equiv j_{n_1}-tq\mod p+q\text{ if }p+q\text{ is odd or }\\
&j_{n_1}\equiv j_{n_t}-q\equiv\cdots\equiv j_{n_2}-(t-1)q\equiv j_{n_1}-tq\mod p+q\text{ if }p+q\text{ is even.}
\end{split}
\]
Then $tq\equiv0\mod p+q$. Since $(q,p+q)=1$, then $t\equiv0\mod p+q$. Therefore $p+q\leq t\leq m<p+q$ which is a contradiction.
\end{proof}

\begin{proof}[Proof of Proposition \ref{pmain}]
First, we prove the existence of such a sequence. It is easy to see by Lemma \ref{lmxatleast} that for every matrix $\mx{B}$ with non-negative elements and $n=0,1$, if the $l$'th column of $\mx{B}$ contains $m$ non-zero elements then the $l$'th column of the matrix $\mx{A}_n\mx{B}$ contains at least $m$ non-zero elements. Moreover, by Lemma \ref{lmxgrowth}, for every column $l$ of $\mx{B}$ there exists an $n\in\left\{0,1\right\}$ such that if it contains $m$ non-zero elements then the $l$'th column of $\mx{A}_n\mx{B}$ contains at least $m+1$ non-zero elements.

Therefore, by taking $n=(p+q)(p+q-1)+1$ we have that there exists $\left\{\xi_k\right\}_{k=1}^{n}$ with $\xi_k\in \{0,1\}$ for which every entry
of the matrix $\mx{A}_{\xi_{n}}\cdots\mx{A}_{\xi_1}$ is non-zero.

For a $(p+q)\times(p+q)$ non-negative matrix $B$ and $1\leq j \leq p+ q$ we let $n_j(\mx{B})$ denote the number of entries that are zero in the $j$'th column.  We observe that for such a matrix we have $n_j(\mx{B}) \geq n_j(\mx{A}_i\mx{B})$ for each $i=0,1$.  Furthermore, we have that there is at most one matrix $\mx{A}_i$ for which $n_j(\mx{B}) = n_j(\mx{A}_i\mx{B})$.  Suppose now that for a finite word $(\xi_1,\dots,\xi_n)$ we have that the matrix $\mx{A}_{\xi_n}\cdots\mx{A}_{\xi_1}$ contains at least one zero entry.  Thus, we have that
\begin{equation*}
1 \leq \sum_{j=1}^{p+q} n_j(\mx{A}_{\xi_n}\cdots\mx{A}_{\xi_1}) \leq \sum_{j=1}^{p+q} n_j(\mx{A}_{\xi_{n-1}}\cdots\mx{A}_{\xi_1}) \leq \cdots \leq \sum_{j=1}^{p+q} n_j(\mx{A}_{\xi_1}) \leq (p+q)(p+q-1).
\end{equation*}
This implies that in this chain of inequalities we must have at most $(p+q)(p+q-1)-1$ strict inequalities. This means at least $n+1-(p+q)(p+q-1)$ of our choice of $\xi_k$ is determined by $\xi_1,\xi_2,\ldots,\xi_{k-1}$.  This implies the inequality (\ref{epmain}).

\end{proof}

It is natural to introduce the dyadic symbolic space. Let $\Xi=\left\{0,1\right\}^{\N}$ and $\Xi^*$ be the set of dyadic finite length words. Define the natural projection $\pi:\Xi\mapsto[0,1]$ by
\[
\pi(\ii)=\sum_{k=1}^{\infty}\frac{i_k}{2^k}.
\]
Moreover, let $\sigma$ be the left shift operator on $\Xi$.

For any $\theta$ with $\tan\theta\in\mathbb{Q}$ and $a\in\Lambda_{\theta}$ let us define $\Gamma_a=\left\{a+\frac{i}{q}\in\Lambda_{\theta}:i\in\mathbb{Z}\right\}$ and $F_{\theta,a}=\bigcup_{b\in\Gamma_a}E_{\theta,b}$.

\begin{prop}\label{pLebtypbox}
Let $p,q\in\N$ be relative primes and let $\theta\in(0,\frac{\pi}{2})$ be such that $\tan\theta=\frac{p}{q}$. Then for Lebesgue-almost every $a\in\Lambda_{\theta}$
\[
\dim_BE_{\theta,a}=\alpha(\theta),
\]
where
\begin{equation}\label{ealpha1}
\alpha(\theta)=\frac{1}{\log2}\lim_{n\rightarrow\infty}\frac{1}{n}\log\underline{e}\mx{A}_{\xi_1}\cdots\mx{A}_{\xi_n}\underline{e},\text{ for $\mathbb{P}$-a.a. }(\xi_1,\xi_2,\dots),
\end{equation}
where $\mathbb{P}$ is the equidistributed Bernoulli measure on $\Xi$. Similarly,
\begin{equation}\label{ealpha2}
\alpha(\theta)=\frac{1}{\log2}\lim_{n\rightarrow\infty}\frac{1}{n}\sum_{\xi_1,\dots,\xi_n}\frac{1}{2^n}\log\underline{e}\mx{A}_{\xi_1}\cdots\mx{A}_{\xi_n}\underline{e}.
\end{equation}
\end{prop}

\begin{proof}
Since $\mx{A}_0, \mx{A}_1$ are non-negative matrices, we have for any $(\xi_1,\dots,\xi_n)\in\Xi^*$ and $1\leq k\leq n$
\[
\underline{e}\mx{A}_{\xi_1}\cdots\mx{A}_{\xi_n}\underline{e}\leq\underline{e}\mx{A}_{\xi_1}\cdots\mx{A}_{\xi_k}\underline{e}\ \underline{e}\mx{A}_{\xi_{k+1}}\cdots\mx{A}_{\xi_n}\underline{e}.
\]

Let $\mathbb{P}=\left\{\frac{1}{2},\frac{1}{2}\right\}^{\N}$ be the equidistributed Bernoulli measure on $\Xi$. Then by the sub-additive ergodic theorem (see \cite[p. 231]{W}) we have for $\mathbb{P}$-almost all $\underline{\xi}\in\Xi$ the limit (\ref{ealpha1}) exists and is  constant. The equation (\ref{ealpha2}) follows also from the sub-additive ergodic theorem.

It is easy to see that the measure $\sum_{k=1}^{p+q}\frac{1}{p+q}\left.\mathbb{P}\circ\pi^{-1}\circ h_k\right|_{I_k}$ is equivalent with the Lebesgue measure on $\Lambda_{\theta}$, where $h_k(x)=-qx+q-k$, so that $h_k(I_k)=[0,1]$. This and Lemma \ref{lmxbox} implies that for Lebesgue almost every $a\in\Lambda_{\theta}$
\begin{equation}\label{edim1}
\max_{b\in\Gamma_a}\dim_BE_{\theta,b}=\dim_BF_{\theta,a}=\alpha(\theta).
\end{equation}

Let $(\xi_1,\dots,\xi_{n_0})\in\left\{0,1\right\}^{n_0}$ be as in Proposition \ref{pmain}. Then for every $1\leq k\leq p+q$ and every finite length word $(\zeta_1,\dots,\zeta_n)\in\left\{0,1\right\}^*$ and Lebesgue-almost every $a\in I_k^{\zeta_1,\dots,\zeta_n\xi_1\dots\xi_{n_0}}$ we have
\[
\dim_BE_{\theta,a}=\dim_BF_{\theta,a'}=\alpha(\theta),
\]
where $a'=2^{n+n_0}\left(a-1+\frac{k-1}{q}\right)+\frac{1}{q}\sum_{i=1}^n2^{n+n_0-i}\zeta_i+\frac{1}{q}\sum_{i=1}^{n_0}2^{n_0-i}\xi_i+1-\frac{k-1}{q}$. The statement of the proposition follows from the fact that the set\\ $\bigcup_{k=1}^{p+q}\bigcup_{n=0}^{\infty}\bigcup_{(\zeta_1,\dots,\zeta_n)\in\left\{0,1\right\}^n}I_k^{\zeta_1,\dots,\zeta_n\xi_1\dots\xi_{n_0}}$ has full Lebesgue measure in $\Lambda_{\theta}$.
\end{proof}

\begin{lemma}\label{lalpha}
The function $\alpha(\theta)<s-1$ for every $\theta$ such that $\tan\theta\in\mathbb{Q}^{+}$.
\end{lemma}

The proof of Lemma \ref{lalpha} coincides with the proof of \cite[Theorem 9]{MS}, (see \cite[Subsection 3.4, Subsection 3.5]{MS}), therefore we omit it.

Finally, we have to state a proposition about the coincidence of the Hausdorff and box dimension for ``typical'' points before we prove Theorem \ref{ttyp}.

\begin{prop}\label{phdbdcoin}
Let $p,q\in\N$ be relative primes and let $\theta\in(0,\frac{\pi}{2})$ be such that $\tan\theta=\frac{p}{q}$. Let $\eta$ be a left shift invariant measure on $\Xi$ such that
\begin{equation}\label{eassumption}
\eta\left(\bigcup_{n=0}^{\infty}\bigcup_{(\zeta_1,\dots,\zeta_n)\in\left\{0,1\right\}^n}[\zeta_1,\dots,\zeta_n\xi_1\dots\xi_{n_0}]\right)=1,
\end{equation}
where $(\xi_1,\dots,\xi_{n_0})$ is as in Proposition \ref{pmain}. Let $\eta=\sum_{k=1}^{p+q}\eta_k$ be an arbitrary positive decomposition of $\eta$. (That is, $\eta_k([\zeta_1,\dots,\zeta_n])>0$ for any $1\leq k\leq p+q$ and any cylinder set.) Then for $\lambda$-almost every $a\in\Lambda_{\theta}$
\[
\dim_HE_{\theta,a}=\dim_BE_{\theta,a},
\]
where
\[
\lambda=\sum_{k=1}^{p+q}\left.\eta_k\circ\pi^{-1}\circ h_k\right|_{I_k}.
\]
\end{prop}

The following lemma appears in a paper of Kenyon and Peres \cite[Proposition 2.6]{KP}, the proof is attributed to Ledrappier.  We state the lemma only for our special case.

\begin{lemma}[Ledrappier]\label{lboxhd}
Let $T_2$ be the endomorphism $T_2(x)=2x\mod1$ on the one-dimensional torus $S^1$. Assume that $F\subset S^1\times S^1=\mathbb{T}^2$ is compact and invariant under $T_2\times T_2$ and $\nu$ a $T_2$-invariant probability measure on $S^1$. Then for $\nu$-a.e. $x$
\[
\dim_H\proj^{-1}(x)=\dim_B\proj^{-1}(x),
\]
where $\proj:F\mapsto S^1$ is the projection to the second coordinate.
\end{lemma}

\begin{proof}[Proof of Proposition \ref{phdbdcoin}]
It is easy to see that
\[
F_{\theta,a}=\Lambda\cap\left\{(x,y):px-qy\equiv -qa\mod1\right\}.
\]
Let $P:(x,y)\mapsto(x,(px-qy)\mod1)$ be a map of $\mathbb{T}^2$ into itself. Then
\[
\underline{\dim}_BP(F_{\theta,a})=\underline{\dim}_BF_{\theta,a},\ \overline{\dim}_BP(F_{\theta,a})=\overline{\dim}_BF_{\theta,a}\text{ and }\dim_HP(F_{\theta,a})=\dim_HF_{\theta,a}.
\]
and $P(\Lambda)\subset\mathbb{T}^2$ is compact and $T_2\times T_2$-invariant. Moreover, let $Q(a)=-qa\mod1$ be the mapping $\Lambda_{\theta}$ into $S^1$. Since $\eta$ is left shift invariant then $\lambda\circ Q^{-1}=\eta\circ\pi^{-1}$ is $T_2$ invariant. Since
\[
\proj^{-1}(-qa\mod1)=P(F_{\theta,a})
\]
by Lemma \ref{lboxhd} we have for $\lambda$-almost all $a\in\Lambda_{\theta}$ that
\begin{equation}\label{eprcoin}
\dim_HF_{\theta,a}=\dim_BF_{\theta,a}.
\end{equation}

Let $(\xi_1,\dots,\xi_{n_0})\in\left\{0,1\right\}^{n_0}$ be as in Proposition \ref{pmain}. Then by assumptions we have that for every $1\leq k\leq p+q$ and every finite length word $(\zeta_1,\dots,\zeta_n)\in\left\{0,1\right\}^*$ the measure $\lambda(I_k^{\zeta_1,\dots,\zeta_n\xi_1\dots\xi_{n_0}})>0$ and for $\lambda$-almost every $a\in I_k^{\zeta_1,\dots,\zeta_n\xi_1\dots\xi_{n_0}}$ the equation (\ref{eprcoin}) holds. Moreover, the fact that the matrix $\mx{A}_{\xi_1}\cdots\mx{A}_{\xi_{n_0}}$ have strictly positive coefficient implies that
\[
\dim_BE_{\theta,a}=\dim_BF_{\theta,a'}=\dim_HF_{\theta,a'}=\dim_HE_{\theta,a},
\]
where $a'=2^{n+n_0}\left(a-1+\frac{k-1}{q}\right)+\frac{1}{q}\sum_{i=1}^n2^{n+n_0-i}\zeta_i+\frac{1}{q}\sum_{i=1}^{n_0}2^{n_0-i}\xi_i+1-\frac{k-1}{q}$. The proof is completed by applying the assumption (\ref{eassumption}).
\end{proof}

\begin{proof}[Proof of Theorem \ref{ttyp}]
Theorem \ref{ttyp}(\ref{ttyp1}) is an easy consequence of Proposition \ref{pLebtypbox}, Lemma \ref{lalpha} and Proposition \ref{phdbdcoin}.

The equalities of Theorem \ref{ttyp}(\ref{ttyp2}) follow from Corollary \ref{cbox} and Proposition \ref{phdbdcoin}. It is enough to prove that $\beta(\theta)>s-1$. To prove this fact, we use the method of \cite{R}.

Define $\eta$ probability measure on $\Xi$ as \[
\eta([\xi_1,\dots,\xi_n])=\frac{1}{3^n}\underline{e}\mx{A}_{\xi_1}\cdots\mx{A}_{\xi_n}\underline{p},
\]
where $\underline{p}$ is the unique probability vector such that $\frac{1}{3}\left(\mx{A}_0+\mx{A}_1\right)\underline{p}=\underline{p}$. Then it is easy to see that $\eta$ is left shift invariant, moreover, by Perron-Frobenius Theorem, $\eta$ is mixing and therefore, an ergodic probability measure. Decompose $\eta=\sum_{k=1}^{p+q}\eta_k$ as
\[
\eta_k([\xi_1,\dots,\xi_n])=\frac{1}{3^n}\underline{e}_k\mx{A}_{\xi_1}\cdots\mx{A}_{\xi_n}\underline{p}
\]
for every cylinder set $[\xi_1,\dots,\xi_n]$. Let us recall that $\nu_{\theta}$ is the projection of the natural self-similar measure on $\Lambda$. Observe that $\left.\nu_{\theta}\right|_{I_k}\circ h_k=\eta_k\circ\pi^{-1}$ and define $\widetilde{\nu}_\theta(.)=\sum_{k=1}^{p+q}\left.\nu_{\theta}\right|_{I_k}\circ h_k=\eta\circ\pi^{-1}$. Then $\widetilde{\nu}_{\theta}$ is a $T_2$- invariant probability measure satisfying the assumptions of Proposition \ref{phdbdcoin}.

By the Volume lemma \cite[Theorems 10.4.1, 10.4.2]{PU} we have
\begin{equation}\label{elimbeta}
\dim_H\widetilde{\nu}_{\theta}=\lim_{n\rightarrow\infty}-\frac{1}{n\log2}\sum_{\xi_1,\dots,\xi_n=0}^1\frac{1}{3^n}\underline{e}\mx{A}_{\xi_1}\cdots\mx{A}_{\xi_n}\underline{p}\log\left(\frac{1}{3^n}\underline{e}\mx{A}_{\xi_1}\cdots\mx{A}_{\xi_n}\underline{p}\right)
\end{equation}
On the other hand, since $\left.\nu_{\theta}\right|_{I_k}\circ h_k\ll\widetilde{\nu}_{\theta}$ for every $1\leq k\leq p+q$ which implies that $\dim_H\left.\nu_{\theta}\right|_{I_k}=\dim_H\left.\nu_{\theta}\right|_{I_k}\circ h_k\leq\dim_H\widetilde{\nu}_{\theta}$. However, \[
\dim_H\widetilde{\nu}_{\theta}=\inf_{1\leq k\leq p+q}\dim_H\left.\nu_{\theta}\right|_{I_k}\circ h_k=\inf_{1\leq k\leq p+q}\dim_H\left.\nu_{\theta}\right|_{I_k}=\dim_H\nu_{\theta}.
\]

By Lemma \ref{lalpha} there exists $\delta>0$ such that for sufficiently large $n$ there exists a sequence $(\xi_1,\dots,\xi_n)$ with
\[
\underline{e}\mx{A}_{\xi_1}\cdots\mx{A}_{\xi_n}\underline{p}<\frac{1}{2^{n+\delta n}}.
\]
This implies that the limit in (\ref{elimbeta}) is strictly less than $1$. The proof can be finished by Corollary \ref{cbox}.
\end{proof}

\begin{proof}[Proof of Proposition \ref{pcalc}]
The statement of the proposition follows from Proposition \ref{pLebtypbox} and the proof of Theorem \ref{ttyp}(\ref{ttyp2}).
\end{proof}

\section{Proof of Theorem \ref{tspectra}}\label{stspectra}

In this section we apply the results of \cite{F1,F2,FL2} to the matrices $A_0,A_1$ to obtain a multifractal description of the dimension of the slices.  Let \[
\widetilde{\Lambda}_{\theta}=\left\{a=1-\frac{k-1}{q}-\frac{1}{q}\sum_{i=1}^{\infty}\frac{\xi_i}{2^i}\in\Lambda_{\theta}:\exists k\geq1,\  \mx{A}_{\xi_1}\cdots\mx{A}_{\xi_k}>0\right\}.
\]
By (\ref{epmain}) we have
\begin{equation}\label{eb0}
\overline{\dim}_B\Lambda_{\theta}\backslash\widetilde{\Lambda}_{\theta}=0.
\end{equation}
Moreover, we can reformulate Lemma \ref{lmxbox}.

\begin{lemma}\label{lmxbox2}
Let $\theta$ and $a\in\widetilde{\Lambda}_{\theta}$ be such that $\tan\theta=\frac{p}{q}$ and
\[
a=1-\frac{k-1}{q}-\frac{1}{q}\sum_{i=1}^{\infty}\frac{\xi_i}{2^i}
\]
then
\[
\underline{\dim}_B E_{\theta,a}=\liminf_{n\rightarrow\infty}\frac{\log\underline{e}\mx{A}_{\xi_1}\cdots\mx{A}_{\xi_n}\underline{e}}{n\log2}\text{ and }\overline{\dim}_B E_{\theta,a}=\limsup_{n\rightarrow\infty}\frac{\log\underline{e}\mx{A}_{\xi_1}\cdots\mx{A}_{\xi_n}\underline{e}}{n\log2}.
\]
\end{lemma}

\begin{proof}[Proof of Proposition \ref{cspectra}(\ref{tspectra1})]
As a consequence of Lemma \ref{lmxbox2} and (\ref{eb0}) we have
\begin{multline*}
\dim_H\left\{a\in\Lambda_{\theta}:\dim_BE_{\theta,a}=\alpha\right\}=\\
\dim_H\left\{1-\frac{k-1}{q}-\frac{1}{q}\sum_{i=1}^{\infty}\frac{\xi_i}{2^i}\in\widetilde{\Lambda}_{\theta}:\lim_{n\rightarrow\infty}\frac{\log\underline{e}\mx{A}_{\xi_1}\cdots\mx{A}_{\xi_n}\underline{e}}{n}=\alpha\log2\right\}=\\
\dim_H\left\{(\xi_1,\xi_2,\dots)\in\Xi:\lim_{n\rightarrow\infty}\frac{\log\underline{e}\mx{A}_{\xi_1}\cdots\mx{A}_{\xi_n}\underline{e}}{n}=\alpha\log2\right\}.
\end{multline*}
By Proposition \ref{pmain}, one can finish the proof using \cite[Theorem 1.1]{F2}.
\end{proof}

By \cite[Lemma 2.2]{F2} and \cite[Theorem 3.3]{FL2} we state the following lemma for the pressure function.

\begin{lemma}\label{lpresproperty}
Let $P(t)$ be defined as in (\ref{emxpressure}). Then $P(t)$ is monotone increasing, convex and continuous for $t\in\R$. Moreover, for $t>0$ the pressure is differentiable.
\end{lemma}

\begin{lemma}\label{lgamma1}
For every $0\leq\delta\leq\alpha(\theta)$,
\[
\dim_H\left\{a\in\Lambda_{\theta}:\dim_HE_{\theta,a}\geq\delta\right\}=1.
\]
\end{lemma}

\begin{proof}
For every $0\leq\delta\leq\alpha(\theta)$ we have
\begin{multline*}
\dim_H\left\{a\in\Lambda_{\theta}:\dim_HE_{\theta,a}\geq\delta\right\}\geq\\\dim_H\left\{a\in\Lambda_{\theta}:\dim_HE_{\theta,a}=\dim_BE_{\theta,a}=\alpha(\theta)\right\}=1.
\end{multline*}
The last equation follows from Theorem \ref{ttyp}(\ref{ttyp1}). The upper bound is trivial.
\end{proof}

\begin{lemma}\label{lpreslim}
Let $P(t)$ be defined as in (\ref{emxpressure}). Then
\[
\lim_{t\rightarrow0+}P'(t)=\alpha(\theta)\log2.
\]
\end{lemma}

\begin{proof} First, we prove $\lim_{t\rightarrow0+}P'(t)\geq\alpha(\theta)\log2$. Suppose by way of contradiction that that there is a $t'>0$ such that $P'(t')=\alpha(\theta)\log2$ and that for every $0<t<t'$ we have $P'(t)<\alpha(\theta)\log2$. Then
\[
1=\dim_H\left\{a\in\Lambda_{\theta}:\dim_BE_{\theta,a}=\alpha(\theta)\right\}=\inf_t\left\{-\alpha(\theta)t+\frac{P(t)}{\log2}\right\}=-\alpha(\theta)t'+\frac{P(t')}{\log2}.
\]
Therefore $P(0)=\log2$ and $P(t')=\log2\alpha(\theta)t'+\log2$ contradicting our assumption that $P'(t)<\alpha(\theta)\log2$.

We now prove the other inequality  $\lim_{t\rightarrow0+}P'(t)\leq\alpha(\theta)\log2$.  Suppose now that $\lim_{t\rightarrow0+}P'(t)>\delta\log(2)>\alpha(\theta)\log(2)$ for some $\delta$.  Then by Theorem \ref{cspectra}(\ref{tspectra1})  there is a $t^-\leq0$
\begin{multline*}
\dim_H\left\{a\in\Lambda_{\theta}:\dim_BE_{\theta,a}=\delta\right\}=\inf_t\left\{-\delta t+\frac{P(t)}{\log2}\right\}=-\delta t^-+\frac{P(t^-)}{\log2}>\\-\alpha(\theta)t^-+\frac{P(t^-)}{\log2}\geq\inf_t\left\{-\alpha(\theta)t+\frac{P(t)}{\log2}\right\}=\dim_H\left\{a\in\Lambda_{\theta}:\dim_BE_{\theta,a}=\alpha(\theta)\right\}=1,
\end{multline*}
which is a contradiction. (The last equality follows from Theorem \ref{ttyp}(\ref{ttyp1}).)
\end{proof}

Before we prove the case when $\alpha(\theta)<\delta\leq b_{\max}$ we need the so-called Gibbs measure.

\begin{lemma}\label{lgibbs}
For every $t>0$ there is a unique ergodic, left shift invariant Gibbs measure $\mu_t$ on $\Xi$ such that there exists a $C>0$ that for any $(\xi_1,\dots,\xi_k)\in\Xi^{*}$
\[
C^{-1}\leq\frac{\mu_t((\xi_1,\dots,\xi_k))}{\left(\underline{e}\mx{A}_{\xi_1}\cdots\mx{A}_{\xi_k}\underline{e}\right)^te^{-kP(t)}}\leq C.
\]
Moreover,
\begin{equation}\label{lgibbs2}
\dim_H\mu_t=\frac{-tP'(t)+P(t)}{\log2}
\end{equation}
and
\begin{equation}\label{lgibbs3}
\lim_{n\rightarrow\infty}\frac{\log\underline{e}\mx{A}_{\xi_1}\cdots\mx{A}_{\xi_n}\underline{e}}{n\log2}=\frac{P'(t)}{\log2}\text{ for $\mu_t$-a.a. }(\xi_1,\xi_2,\dots).
\end{equation}
\end{lemma}

The proof of the lemma follows from \cite[Theorem 3.2]{FL2} and \cite[Proof of Theorem 1.3]{FL2}.

\begin{lemma}\label{lgamma2}
For every $\alpha(\theta)<\delta\leq b_{\max}$,
\[
\dim_H\left\{a\in\Lambda_{\theta}:\dim_HE_{\theta,a}\geq\delta\right\}=\inf_{t>0}\left\{-\delta t+\frac{P(t)}{\log2}\right\}.
\]
\end{lemma}

\begin{proof}
Let us observe by Lemma \ref{lpreslim} that
\[
\inf_{t}\left\{-\delta t+\frac{P(t)}{\log2}\right\}=\inf_{t>0}\left\{-\delta t+\frac{P(t)}{\log2}\right\}.
\]
First, we will prove the upper bound with the method of \cite[Lemma 3.18]{Wi}.
Let us define the following set of intervals:
\[
\mathbf{A}_n(\varepsilon)=\left\{(\xi_1,\dots,\xi_k):k\geq n,\ \delta-\varepsilon\leq\frac{\log\underline{e}\mx{A}_{\xi_1}\cdots\mx{A}_{\xi_k}\underline{e}}{k\log2}\right\}.
\]
It is easy to see that the set
\[
\bigcup_{j=1}^{p+q}\bigcup_{(\xi_1,\dots,\xi_k)\in\mathbf{A}_n(\varepsilon)}I_j^{\xi_1,\dots,\xi_k}
\]
covers the set $G_{\delta}=\left\{a\in\Lambda_{\theta}:\delta\leq\underline{\dim}_BE_{\theta,a}\right\}$. Let $\mathbf{B}_n(\varepsilon)$ be the set of disjoint cylinders of $\mathbf{A}_n(\varepsilon)$ such that
\[
\bigcup_{j=1}^{p+q}\bigcup_{(\xi_1,\dots,\xi_k)\in\mathbf{B}_n(\varepsilon)}I_j^{\xi_1,\dots,\xi_k}=\bigcup_{j=1}^{p+q}\bigcup_{(\xi_1,\dots,\xi_k)\in\mathbf{A}_n(\varepsilon)}I_j^{\xi_1,\dots,\xi_k}.
\]
Then for any $t>0$ and $\varepsilon'>0$ we have
\begin{multline*}
\mathcal{H}_{2^{-n}}^{-\delta t+\frac{P(t)}{\log2}+\varepsilon't}(G_{\delta})\leq\sum_{j=1}^{p+q}\sum_{(\xi_1,\dots,\xi_k)\in\mathbf{B}_n(\varepsilon)}\left|I_j^{\xi_1,\dots,\xi_k}\right|^{-\delta t+\frac{P(t)}{\log2}+\varepsilon't}\leq\\(p+q)2^{(\varepsilon-\varepsilon')n t}\sum_{(\xi_1,\dots,\xi_k)\in\mathbf{B}_n(\varepsilon)}\left(\underline{e}\mx{A}_{\xi_1}\cdots\mx{A}_{\xi_k}\underline{e}\right)^te^{-kP(t)}.
\end{multline*}
By Lemma \ref{lgibbs}
\[
\mathcal{H}_{2^{-n}}^{-\delta t+\frac{P(t)}{\log2}+\varepsilon'}(G_{\delta})\leq C(p+q)2^{(\varepsilon-\varepsilon')nt}\sum_{(\xi_1,\dots,\xi_k)\in\mathbf{B}_n(\varepsilon)}\mu_t((\xi_1,\dots,\xi_k))\leq C(p+q)2^{(\varepsilon-\varepsilon')nt}.
\]
This implies that
\[
\dim_H\left\{a\in\Lambda_{\theta}:\delta\leq\dim_HE_{\theta,a}\right\}\leq\dim_H\left\{a\in\Lambda_{\theta}:\delta\leq\underline{\dim}_BE_{\theta,a}\right\}\leq-\delta t+\frac{P(t)}{\log2}+\varepsilon't
\]
for any $t>0$ and $\varepsilon'>\varepsilon>0$. This proves the upper bound.

Now, we prove the lower bound. By Lemma \ref{lpresproperty}, for every $\alpha(\theta)<\delta< b_{\max}$ there exists a $t>0$ such that $P'(t)=\delta\log2$. By Lemma \ref{lgibbs}, let $\mu_t$ be the Gibbs measure. The measure $\mu_t$ is shift invariant and ergodic. Moreover, by the Gibbs property, $\mu_t$ satisfies the assumption of Proposition \ref{phdbdcoin} and we have
\[
\dim_HE_{\theta,a}=\dim_BE_{\theta,a}\text{ for $\mu_t$-almost all $(\xi_1,\xi_2,\dots)$,}
\]
where $a=1-\frac{k-1}{q}-\frac{1}{q}\sum_{i=1}^{\infty}\frac{\xi_i}{2^i}$ for some $1\leq k\leq p+q$.
Then by (\ref{lgibbs2}) and (\ref{lgibbs3}) we have
\begin{multline*}
\dim_H\left\{a\in\Lambda_{\theta}:\dim_HE_{\theta,a}\geq\delta\right\}\geq\dim_H\left\{a\in\Lambda_{\theta}:\dim_HE_{\theta,a}=\dim_BE_{\theta,a}=\delta\right\}\geq\\\dim_H\mu_t=-t\delta+\frac{P(t)}{\log2}\geq\inf_{t>0}\left\{-t\delta+\frac{P(t)}{\log2}\right\}.
\end{multline*}
If $\delta=b_{\max}$ then
\begin{multline*}
\dim_H\left\{a\in\Lambda_{\theta}:\dim_HE_{\theta,a}\geq b_{\max}\right\}\leq\lim_{\delta\rightarrow b_{\max}+}\dim_H\left\{a\in\Lambda_{\theta}:\dim_HE_{\theta,a}\geq\delta\right\}=\\\lim_{\delta\rightarrow b_{\max}+}\inf_{t>0}\left\{-t\delta+\frac{P(t)}{\log2}\right\}=\inf_{t>0}\left\{-t b_{\max}+\frac{P(t)}{\log2}\right\}=0.
\end{multline*}
In the last two equations we used the continuity property \cite[Theorem 1.1]{F2} and the definition of $b_{\max}$.
\end{proof}

\begin{proof}[Proof of Theorem \ref{tspectra}(\ref{tspectra2})]
The proof is the combination of Lemma \ref{lgamma1} and Lemma \ref{lgamma2}.
\end{proof}

\begin{proof}[Proof of Theorem \ref{tspectra}(\ref{tspectra3})]
By the observation
\begin{multline*}
\dim_H\left\{a\in\Lambda_{\theta}:\underline{\dim}_BE_{\theta,a}\geq\delta\right\}\geq\dim_H\left\{a\in\Lambda_{\theta}:\dim_HE_{\theta,a}=\delta\right\}\geq\\\dim_H\left\{a\in\Lambda_{\theta}:\dim_BE_{\theta,a}=\dim_HE_{\theta,a}=\delta\right\}
\end{multline*}
one can finish the proof as Lemma \ref{lgamma2}.
\end{proof}

\noindent{\bf Acknowledgment.} The authors would like to express their gratitude to the anonymous referees for their reading of the original version as well as their helpful comments.
\\\\
The research of
B\'ar\'any and Simon was supported by OTKA Foundation grant \# K 71693. Ferguson acknowledges support from EPSRC grant EP/I024328/1 and the University of Bristol.

\bibliographystyle{alpha}
\bibliography{sierbib}

\end{document}